\documentclass[a4paper,12pt,reqno]{amsart}
\usepackage{amssymb}
\usepackage{latexsym}
\usepackage{amsfonts}
\usepackage{amsmath}
\usepackage{amsthm}
\usepackage{multicol}
\usepackage{mathrsfs}
\theoremstyle{definition}
\newtheorem{definition}{Definition}[section]
\newtheorem{proposition}[definition]{Proposition}
\newtheorem{theorem}[definition]{Theorem}

\newtheorem{lemma}[definition]{Lemma}
\newtheorem{remark}[definition]{Remark}

\def\N{{\mathbb{N}}}

\def\R{{\mathbb{R}}}

\def\<{\mathop{<}}
\def\>{\mathop{>}}

\newcommand{\spt}{\mathrm{spt}\,}
\newcommand{\divergence}{\mathrm{div}\,}
\newcommand{\dist}{\mathrm{dist}\,}

\numberwithin{equation}{section}

\def\XXint#1#2#3{{\setbox0=\hbox{$#1{#2#3}{\int}$}
\vcenter{\hbox{$#2#3$}}\kern-.5\wd0}}

\makeatletter
\@namedef{subjclassname@2020}{\textup{2020} Mathematics Subject Classification}
\makeatother

\title[On obstacle problem for Brakke's mean curvature flow]
{On obstacle problem for \\
Brakke's mean curvature flow}
\author[K. Takasao]{Keisuke Takasao}
\address{Department of Mathematics/Hakubi Center, Kyoto University,
Kitashirakawa-Oiwakecho, Sakyo, Kyoto 606-8502, Japan}
\email{k.takasao@math.kyoto-u.ac.jp}
\keywords{mean curvature flow, Allen-Cahn equation, phase field method}
\subjclass[2020]{Primary~35K93, Secondary~53E10}
\thanks{}
\date{}
\begin{document}
\maketitle
\begin{abstract}
We consider the obstacle problem of the weak solution for the mean curvature flow, 
in the sense of Brakke's mean curvature flow.
We prove the global existence of the weak solution
with obstacles which have $C^{1,1}$ boundaries, in two and three space dimensions.
To obtain the weak solution, 
we use the Allen-Cahn equation with forcing term.
\end{abstract}

\section{Introduction}
Let $T>0$ and $d\geq 2$ be an integer. 
Assume  that $U_t \subset \mathbb{R}^d$ is a bounded open set and 
$M_t$ is a smooth boundary of $U_t$ for any $t \in[0,T)$. 
We call the family of the hypersurfaces $\{ M_t\}_{t \in [0,T)}$ the mean curvature flow if
\begin{equation}
v=h \qquad \text{on} \ M_t ,\ t \in (0,T).
\label{mcf}
\end{equation} 
Here, $v$ and $h$ are the normal velocity vector and the mean curvature vector of $M_t$, 
respectively.
%
%
Brakke~\cite{brakke} proved the global existence of the multi-phase weak solution 
to \eqref{mcf} called Brakke's mean curvature flow.
However, since the flow is defined by 
an integral inequality, its solution may become an empty set after a certain time.
Subsequently, Kim-Tonegawa~\cite{kim-tonegawa} 
proved the global existence of non-trivial
Brakke's mean curvature flow, 
by showing that the each volume of the multi-phase is continuous with respect to $t$.
The phase-field method and the elliptic regularization 
by Ilmanen~\cite{ilmanen1993, ilmanen1994} are known as
another proofs of the global existence of the Brakke's mean curvature flow.
Similar to the Brakke's mean curvature flow, 
the weak solution called $L^2$-flow was studied by 
Mugnai-R{\"o}ger~\cite{mugnai-roger2008, mugnai-roger2011}.
In addition, the regularity of Brakke's mean curvature flow is studied by 
Brakke~\cite{brakke}, White~\cite{white2005}, 
Kasai-Tonegawa~\cite{kasai-tonegawa}, and Tonegawa~\cite{tonegawa2014}.
About results for other types of weak solutions, 
the existence theorem of the viscosity solutions via the level set method
was presented independently by Chen-Giga-Goto~\cite{chen-giga-goto} 
and Evans-Spruck~\cite{evans-spruck1991} at the same period,
and a weak solution using a variational method
was studied by Almgren-Taylor-Wang~\cite{almgren-taylor-wang} 
and Luckhaus-Sturzenhecker~\cite{luckhaus-sturzenhecker}.

\bigskip

Let $O _+ $ and $O_ -$ be open sets with $\dist (O_+, O_-) >0$.
In this paper, we consider the weak solution to \eqref{mcf} 
with the obstacles $O_+$ and $O_-$,
namely, a family of open sets $\{ U_t \}_{t \in [0,T)}$
satisfies $O_+ \subset U_t$ and $U_t \cap O_-=\emptyset$ 
for any $t \in [0,T)$, and the boundary $M_t =\partial U_t$ satisfies \eqref{mcf}
on $( \overline{O_+ \cup O_-} )^c$, in the sense of Brakke's mean curvature flow.
Since the mean curvature flow
can be regarded as a simple model of the cell motility,
it is natural to consider its obstacle problem
(see \cite{elliott-stinner-venkataraman, m-b-r-z}). 
In addition, 
the obstacle problems for elliptic equations including the minimal surface equation
have been studied over a long period of time 
(see \cite{MR1658612, MR2962060, MR880369, MR2270163} and references therein).

About the obstacle problem for the mean curvature flow,
Almeida-Chambolle-Novaga~\cite{almeida-chambolle-novaga} 
showed the global existence of weak solutions for $d\geq 2$ by a variational method.
Moreover, they proved the short time existence and uniqueness of $C^{1,1}$ 
solutions for $d=2$,
when the obstacle has a compact $C^{1,1}$ boundary.
Mercier-Novaga~\cite{mercier-novaga} extended the short time existence and uniqueness of $C^{1,1}$ 
solutions for $d\geq 2$, and they also proved the global existence and uniqueness
of the graphical viscosity solutions if the boundaries of obstacles are also graphs.
In the case of the viscosity solution with the level set method, 
Mercier~\cite{mercier} showed the global existence and uniqueness of continuous viscosity solutions to
\[
u_t + F(\nabla u, \nabla ^2 u) +k|\nabla u|=0 \quad \text{on} \ \{ u^- \leq u \leq u^+ \},
\]
where $u^-$ and $u^+$ are given uniformly continuous functions with $u^-\leq u^+$, 
$k$ is a given Lipschitz function, 
and the assumptions of $F$ allow this equation to be
the mean curvature flow with forcing term $k$,
in the sense of the level set method.
Ishii-Kamata-Koike~\cite{ishii-kamata-koike} proved the global existence and uniqueness of Lipschitz viscosity
solutions when $k\equiv 0$ and $u^\pm$, $\partial_t u^\pm $, $ \partial _{x_k } u^\pm $,  
$ \partial _{x_k x_l} u^\pm \in L ^\infty (\mathbb{T}^d \times [0,\infty))$
for any $1\leq k,l \leq d$, where $\mathbb{T}^d=(\mathbb{R}/\mathbb{Z})^d$.
Giga-Tran-Zhang~\cite{giga-tran-zhang} studied the 
large time behavior of viscosity solutions 
with constant driving force $k$. 

Let $d = 2$ or $3$, $\Omega :=\mathbb{T}^d$, 
and the obstacles $O _+ , O_ - \subset \Omega$ have $C^{1,1}$ boundaries and
satisfy $\dist (O_+, O_-) >0$. 
In this paper, we prove the global existence of the weak solution to \eqref{mcf}
with obstacles
in the sense of Brakke (see Theorem \ref{thm5.3} below).
Note that the weak solution obtained in this paper has similar properties 
to the weak solution by the minimizing movement in \cite[Theorem 4.6]{almeida-chambolle-novaga}
(see Remark \ref{remark5.3} below).
However, since the uniqueness of the flow we obtain is not known, 
it is an open question whether Brakke's mean curvature flow coincides with
the weak solution studied in \cite{almeida-chambolle-novaga}.

To obtain the result, we use the phase-field method.
Bretin-Perrier~\cite{bretin-perrier} studied the Allen-Cahn equation 
with a penalized double well potential depending on the obstacles.
In contrast, roughly speaking, the Allen-Cahn equation 
considered in this paper (see \eqref{ac} bellow) 
is formally an approximation to the following:
\begin{equation}
v=h + gn, \qquad \text{on} \ M_t ,\ t \in (0,\infty),
\label{eq1.2}
\end{equation}
where $n$ is the outward unit normal vector of $M_t$ and $g$ is given by
\begin{equation*}
   g (x) = 
   \begin{cases}
    \frac{d}{R_0}, & \text{if} \ x \in \overline{O_+} , \\
    -\frac{d}{R_0}, & \text{if} \ x \in \overline{O_-},\\
    0, & \text{otherwise} ,
  \end{cases}
\end{equation*}
where $R_0$ is given in \eqref{eq:2.3}.
If the solution $M_t$ touches the obstacle at $x$, 
the absolute value of its mean curvature $| h(x,t) |$ is less than $\frac{d}{R_0}$, 
hence the solution cannot move into the obstacle.
Note that this argument was used in Mercier-Novaga~\cite{mercier-novaga}.
In order to use this argument in the phase-field method, 
we give an appropriate forcing term for the Allen-Cahn equation 
and show simple sub and super solutions that correspond to obstacles
(see Lemma \ref{lem4.1} below).

To obtain the convergence of the Allen-Cahn equation to the Brakke flow,
we need to prove that the Radon measure given by the energy of the Allen-Cahn equation
has good properties, such as that it converges to the mass measure of an integral varifold
(see \cite{ilmanen1993}). 
In the case of $d=2$ or $3$, 
R\"{o}ger-Sch\"{a}tzle~\cite{roger-schatzle} proved the properties under the suitable assumptions
for the energies of the Allen-Cahn equation
(this results have been used in 
\cite{liu-sato-tonegawa, mugnai-roger2008, mugnai-roger2011, sato2008, takasao2020}).
The assumption for $d$ in the main result of this paper 
comes from the use of \cite{roger-schatzle} and \cite{mugnai-roger2011}
(see Remark \ref{remark5.4} below). 

\bigskip

The organization of the paper is as follows.
In Section 2, we set out basic definitions and assumptions about the obstacles
and the initial data.
In Section 3 we introduce the Allen-Cahn equation we deal with in this paper. In addition
we also show the standard estimates for the solution.
In Section 4 we give supersolutions and subsolutions to 
the Allen-Cahn equation that are necessary to show that the solutions to
\eqref{mcf} do not intrude into the obstacles.
In Section 5 we prove the global existence of the weak solution to \eqref{mcf}
with obstacles, in the sense of Brakke.
\section{Notation and assumptions}
First we recall some notions and definitions from the geometric measure theory and refer to
\cite{allard, brakke, federer, simon, tonegawa-book} for more details.
Let $d $ be a positive integer.
For $y \in \mathbb{R}^d$ and $r>0$, we define
$B_r (y) := \{ x \in \mathbb{R}^d \, | \, |x-y| <r\}$.
We denote the space of bounded variation functions on 
$U \subset \mathbb{R}^d$ as $BV (U)$. 
For a function $\psi \in BV(U)$, 
we write the total variation measure of the distributional derivative 
$\nabla \psi$ by $\|\nabla \psi\|$. 
Let $\mu$ be a Radon measure on $\mathbb{R}^d$. 
We denote $\mu (\phi) = \int \phi \,d \mu$ for $\phi \in C_c (\mathbb{R}^d)$.
We call $\mu$ $k$-rectifiable ($1\leq k \leq d-1$) 
if $\mu$ is given by $\mu =\theta \mathscr{H}^k \lfloor _{M}$, 
where $M \subset \mathbb{R}^d$ is a $\mathscr{H}^{k}$-measurable
countably $k$-rectifiable set and $\theta \in L_{loc} ^1 (\mathscr{H}^k \lfloor _{M})$
is a positive function $\mathscr{H}^k$-a.e. on $M$.
Especially, if $\theta$ is integer-valued $\mathscr{H}^k$-a.e. on $M$, then 
we say $\mu$ is $k$-integral.
Note that if $M$ is a countably $k$-rectifiable set with locally finite and 
$\mathscr{H}^k$-measurable, then there exists 
the approximate tangent space $T_x M$ for $\mathscr{H}^k$-a.e. $x \in M$.
For $k$-dimensional subspace $S \subset \mathbb{R}^d$
and $g\in C_c ^1 (\mathbb{R}^d ; \mathbb{R}^d)$,
we denote $\divergence_{S} \, g := \sum _{i=1} ^k \nu _i \cdot \nabla _{\nu _i} g$,
where $\{ \nu _1 , \dots, \nu _k \}$ is an orthonormal basis of $S$.
For a rectifiable Radon measure $\mu= \theta \mathscr{H}^k \lfloor_{M}$, $h$ is called a generalized
mean curvature vector if
$$
\int \divergence_{T_x M} \, g \, d \mu = - \int h \cdot g \, d\mu
$$
for any $g \in C_c ^1 (\mathbb{R}^d ; \mathbb{R}^d)$.
The left-hand side is called  the first variation of $\mu$.
The weak solution to the mean curvature flow considered 
in this paper is as follows.
\begin{definition}
Let $U \subset \mathbb{R}^d$ be an open set.
A family of Radon measures $\{ \mu _t \} _{t \in [0,T)}$ on $U$
is called Brakke's mean curvature flow if
\begin{equation}
\int _U \phi \, d\mu _t \Big|_{t=t_1} ^{t_2} 
\leq 
\int _{t_1} ^{t_2} \int _U \{ (\nabla \phi -\phi h) \cdot h + \phi_t \} \, d\mu _t dt  
\label{brakke-ineq}
\end{equation}
for all $0\leq t_1 <t_2 <\infty$ and
$\phi \in C_c ^1 (U\times [0,\infty); [0,\infty))$.
Here $h$ is the generalized mean curvature vector of $\mu _t$.
Note that \eqref{brakke-ineq} is called Brakke's inequality.
\end{definition}

\bigskip

Next we state assumptions for the initial data and the obstacles.
Let $\Omega = \mathbb{T}^d= (\mathbb{R}/\mathbb{Z})^d$,
$O_{+} \subset \Omega$ and $O_{-} \subset \Omega$ be bounded open sets. 
We assume that there exist $R_0 >0$ and $R_1 >0$ such that
\begin{equation}
O_{\pm} = \bigcup _{ B_{R_0} (x) \subset O_{\pm} } B_{R_0} (x)
\qquad \text{(the interior ball condition)}
\label{eq:2.3}
\end{equation}
and $\dist (O_+,O_-)>R_1$. 
Note that if $O_{+}$ and $O_{-}$ have $C^{1,1}$ boundaries, 
then \eqref{eq:2.3} is satisfied for some $R_0 >0$ (see \cite{MR2286038}).
Let $U_0 \subset \Omega$ be a bounded open set and we denote $M _0 :=\partial U_0$.
Throughout this paper, we assume the following:
\begin{enumerate}
\item There exists $\delta_1 >0$ such that $O_+ \subset U_0$ with $\dist (O_+, M_0) >\delta _1$
and $U_0 \subset (O_-)^c$ with $\dist (O_-, M_0) >\delta_1$.
\item There exist $D_0>0$ and $R_2 \in (0,1)$ such that
\begin{equation}
\sup_{x\in \Omega, \, 0<R<R_2 }\frac{\mathscr{H}^{d-1} (M _0  \cap B_R (x)) }{\omega _{d-1}R^{d-1}} \leq D_0 \quad \text{(the upper bounds of the density)}.
\label{initialdata1}
\end{equation}
Here $\omega _{d-1}$ is a $(d-1)$-dimensional volume of the unit ball in $\mathbb{R}^{d-1}$.
\item There exists a family of open sets $\{ U _0 ^i \}_{i=1} ^\infty$ such that $U _0 ^i$ has a  $C^3$ boundary $M _0 ^i$ such that $(U _0 ,M _0)$ be approximated strongly by $\{ (U _0 ^i ,M _0 ^i) \} _{i=1} ^\infty$, that is,
\begin{equation}
\lim _{i\to \infty} \mathscr{L}^d (U _0 \triangle U_0 ^i) =0 \quad \text{and} \quad
\lim _{i \to \infty} \| \nabla \chi _{U ^i _0} \| = \|\nabla \chi _{U _0} \| \ \ \text{as measures.}
\label{initialdata2}
\end{equation}
Moreover
\begin{equation}
\dist (O_{\pm}, M_0 ^i) >\delta _1/2,\qquad \text{for any} \ i \in \N.
\label{initialdata3}
\end{equation}
\end{enumerate}
\begin{remark}\label{remark2.2}
For example, if $U _0$ is a Caccioppoli set, then \eqref{initialdata2} is satisfied (see \cite[Theorem 1.24]{giusti}). 
In addition, if $M_0$ is $C^1$, then \eqref{initialdata1} with $D_0=1+o(R_2)$ and \eqref{initialdata2} hold.
\end{remark}

\section{Allen-Cahn equation with forcing term}
In this section, we consider the Allen-Cahn equation forcing term 
and give basic energy estimates for the solution.

\bigskip

Set $W(s) =(1-s^2)^2/2$ and $q^\varepsilon (r) := \tanh (\frac{r}{\varepsilon})$ 
for $r\in \mathbb{R}$ and $\varepsilon >0$. 
Then $q^{\varepsilon}$ is a solution to
\begin{equation}
\frac{\varepsilon ( q^{\varepsilon} _r)^2 }{2} 
= \frac{W (q ^{\varepsilon})}{\varepsilon} \qquad \text{and} \qquad 
q^{\varepsilon} _{rr}= \frac{W' (q ^{\varepsilon})}{\varepsilon ^2}
\label{q}
\end{equation}
with $q^{\varepsilon}(0)=0, \ q^{\varepsilon}(\pm \infty)=\pm 1$, 
and $ q^{\varepsilon}_r (r)  >0$ for any $r\in \mathbb{R}$. 

Let  $d\geq 2$ and $\{ \varepsilon _i \}_{i=1}^\infty$ be a positive sequence with 
$\varepsilon _i \downarrow 0$ as $i\to \infty$ and $\varepsilon _i \in (0,1)$ for any $i\in \N$
(we often write $\varepsilon _i$ as $\varepsilon$ for simplicity). 
For $U _0^i \subset \Omega$ we define a periodic function $r _0 ^{\varepsilon_i}$ by
\[
   r _0 ^{\varepsilon_i}(x)= \begin{cases}
    \dist(x,M_0 ^i), & \text{if} \ x\in U _0^i ,\\
    -\dist(x,M_0 ^i), & \text{if} \ x\notin U _0 ^i.
  \end{cases}
\]
We remark that $|\nabla r _0 ^{\varepsilon_i} |\leq 1$ a.e. $x\in \Omega$ 
and $r _ 0 ^{\varepsilon_i}$ is smooth near $M _0 ^i$. 
Let $\tilde r _0 ^{\varepsilon_i} \in C^\infty (\Omega)$ be a smoothing of $r_0 ^{\varepsilon_i}$ with 
$|\nabla\tilde r _0 ^{\varepsilon_i} |\leq 1$ and 
$|\nabla ^2 \tilde r _0 ^{\varepsilon_i} |\leq \varepsilon _i ^{-1}$ in $\Omega$, 
and $\tilde r_0 ^{\varepsilon_i} =r _0 ^{\varepsilon_i}$ near $M _0 ^i$.

Define
\begin{equation}
\varphi ^{\varepsilon_i }_0(x):=q^{\varepsilon_i } ( \tilde r_0 ^{\varepsilon_i} (x) ), 
\quad i\geq 1.
\label{initial}
\end{equation}

Let $g^{\varepsilon_i} \in C^\infty  (\Omega)$ be a smooth function such that
\begin{equation}
   g^{\varepsilon _i} (x) = 
   \begin{cases}
    \frac{d}{R_0}, & \text{if} \ \dist (x,O_+) \leq \sqrt{\varepsilon_i} , \\
    -\frac{d}{R_0}, & \text{if} \ \dist (x,O_-) \leq \sqrt{\varepsilon_i},\\
    0, & \text{if} \ \min \{ \dist (x,O_+), \dist (x,O_-) \} \geq 
    2\sqrt{\varepsilon_i},
  \end{cases}
  \label{def-g}
\end{equation}
with $\max _{x \in \Omega} |g^{\varepsilon _i} (x)| \leq \frac{d}{R_0}$, 
$\max _{x \in \Omega} |\nabla g^{\varepsilon _i} (x)| \leq M \varepsilon ^{-1} _i$, and
$\max _{x \in \Omega} |\nabla ^2 g^{\varepsilon _i} (x)| \leq M \varepsilon ^{-2} _i$,
for any $i \in \N$, where $M>0$ is independent of $i$.
To define $g^{\varepsilon_i}$, we may assume that 
$2\sqrt{\varepsilon_i} \leq \frac{R_1}{3}$ for any $i\in \N$, if necessary.

In this paper, we consider the following 
Allen-Cahn equation:
\begin{equation}
\left\{ 
\begin{array}{ll}
\varepsilon_i \varphi ^{\varepsilon_i}_t 
=\varepsilon_i \Delta \varphi ^{\varepsilon_i } 
-\dfrac{W' (\varphi ^{\varepsilon_i })}{\varepsilon_i } 
+ g^{\varepsilon_i} \sqrt{2W(\varphi ^{\varepsilon _i})} ,
& (x,t)\in \Omega \times (0,\infty),  \\
\varphi ^{\varepsilon_i} (x,0) = \varphi _0 ^{\varepsilon_i} (x) ,  &x\in \Omega.
\end{array} \right.
\label{ac}
\end{equation}

\begin{remark}\label{rem3.1}
The definition of the initial data \eqref{initial} implies 
$\max_{x\in \Omega} |\varphi ^\varepsilon _0 (x)| <1 $.
Therefore we have $\sup_{x\in \Omega , t \in [0,T)} |\varphi ^\varepsilon (x,t)| <1 $ 
for the solution $\varphi ^\varepsilon$ to \eqref{ac} and $T>0$, 
by the maximum principle (see \cite{MR0762825}). We give the proof in Proposition \ref{prop3.4}.
Note that the function $\sqrt{2W (\varphi ^\varepsilon)}$ is important in the proof.
By $ |\varphi ^\varepsilon | <1$, we can define $r^\varepsilon=r^\varepsilon (x,t)$ 
by $\varphi ^\varepsilon (x,t) = q ^{\varepsilon} (r ^\varepsilon (x,t))$,
that is, $r ^\varepsilon (x,t) = (q ^{\varepsilon}) ^{-1} (\varphi ^\varepsilon (x,t))$.
\end{remark}

\begin{remark}
\eqref{ac} corresponds to the mean curvature flow
with forcing term \eqref{eq1.2} (see \cite{mugnai-roger2011, soner1995, takasao2020}).
Not only for $|\varphi ^\varepsilon|<1$,
we also need $\sqrt{2W(\varphi ^{\varepsilon})}$ 
to simplify the forcing term when we rewrite \eqref{ac} as an PDE of $r^\varepsilon$
(see \eqref{ac2} below). Furthermore, if we adopt $g ^\varepsilon$ instead of 
$g^\varepsilon \sqrt{2W(\varphi ^{\varepsilon})}$ in \eqref{ac}, 
then the calculation of \eqref{eq5.2} below will fail.
In the case of $g^\varepsilon \equiv 0$, the convergence of \eqref{ac}
to the mean curvature flow with no obstacles is well known
(see \cite{bronsard-kohn, xchen1992, evans-soner-souganidis, ilmanen1993}).
\end{remark}

Here we give the standard pointwise estimate for the solution to \eqref{ac}.

\begin{proposition}\label{prop3.4}
Let $\varphi^\varepsilon $ be a solution to \eqref{ac}.
Then  $\sup_{x\in \Omega , t \in [0,T)} |\varphi ^\varepsilon (x,t)| <1 $ for any $T>0$.
\end{proposition}

\begin{proof}
We only show $\sup_{x \in\Omega ,t \in [0,T)} \varphi ^\varepsilon (x,t) <1$,
because we can obtain $\varphi ^\varepsilon >-1$ similarly. By \eqref{initial}, we have $\sup _{x\in\Omega} |\varphi_0 ^\varepsilon (x) |<1$.
Assume that $\{ t \in [0,T) \, | \, \exists \, x \in \Omega \ \text{s.t.} \ \varphi ^\varepsilon (x,t)=1 \} \not =\emptyset$ and set
$t_0 :=\inf \{ t \in [0,T) \, | \, \exists \, x \in \Omega \ \text{s.t.} \ \varphi ^\varepsilon (x,t)=1 \}$.
Then $t_0 \in (0,T)$ by $\sup _{x\in\Omega} |\varphi_0 ^\varepsilon (x) |<1$ and $\varphi ^\varepsilon$ satisfies
\begin{equation}
\begin{split}
\varphi ^\varepsilon 
\leq & 
\Delta \varphi ^\varepsilon + \left| \frac{2 \varphi ^\varepsilon }{\varepsilon ^2} (1+\varphi ^\varepsilon) \right| (1-\varphi ^\varepsilon)
+ \left| \frac{g}{\varepsilon} (1+\varphi^\varepsilon) \right| (1-\varphi ^\varepsilon)\\
\leq & \Delta \varphi ^\varepsilon + M (1-\varphi ^\varepsilon)
\end{split}
\label{comparison2}
\end{equation}
for any $(x,t) \in \Omega \times (0,t_0)$, where 
$\displaystyle M= \max _{x \in \Omega , t \in [0,T]} \left\{ \left| \frac{2 \varphi ^\varepsilon }{\varepsilon ^2} (1+\varphi ^\varepsilon) \right|
+ \left| \frac{g}{\varepsilon} (1+\varphi^\varepsilon) \right| \right\}$
and we used $1-\varphi ^\varepsilon \geq 0$ in $\Omega \times (0,t_0)$.
We denote $\alpha := \max_{x \in \Omega} \varphi ^\varepsilon _0 (x)$. Note that $\alpha <1$.
Set $\overline \varphi (t) := 1- (1- \alpha ) e^{-Mt}$. Then $\overline \varphi$ 
is monotone increasing and satisfies
\begin{equation}
\overline \varphi _t = \Delta \overline \varphi + M (1-\overline \varphi)
\label{comparison}
\end{equation}
and $\overline \varphi (0) =\alpha \geq \varphi^\varepsilon _0 (x)$ for any $x \in \Omega$.
We remark that $\varphi^\varepsilon $ is a subsolution to \eqref{comparison} in $\Omega \times (0,t_0)$ by \eqref{comparison2}.
Therefore the comparison principle implies 
$\varphi^\varepsilon (x,t) \leq \overline \varphi (t) \leq \overline \varphi (T) <1$ 
for any $(x,t) \in \Omega \times [0,t_0)$. Then we would have a contradiction from 
$\varphi ^\varepsilon (x,t_0) =1$ for some $x \in \Omega$.
Therefore $\varphi ^\varepsilon (x,t) <1$ for any $(x,t) \in \Omega \times [0,T)$
and $\varphi ^\varepsilon$ satisfies \eqref{comparison2} in $\Omega \times (0,T)$.
Using the comparison principle again, we obtain 
$\sup_{x\in \Omega , t \in [0,T)} |\varphi ^\varepsilon (x,t)| \leq \overline \varphi (T)<1$.
\end{proof}

Next we define the measures that correspond to the surface $M_t$ in Section 1.

\begin{definition}\label{def-mu}
Set $\sigma := \int_{-1} ^1 \sqrt{2W(s)} \, ds $. Assume that 
$\varphi ^{\varepsilon_i}$ is a solution to \eqref{ac}. 
We denote Radon measures $\mu _t ^{\varepsilon_i}$, $\tilde \mu _t ^{\varepsilon_i}$,
 and $\hat \mu _t ^{\varepsilon_i}$ by
\begin{equation}
\mu _t ^{\varepsilon_i}(\phi) := \frac{1}{\sigma}\int _{\Omega} 
\phi(x) \Big( \frac{\varepsilon_i |\nabla \varphi ^{\varepsilon_i} (x,t)|^2}{2} 
+ \frac{W (\varphi ^{\varepsilon_i} (x,t))}{\varepsilon_i} \Big) dx,
\label{eq5.1}
\end{equation}
\begin{equation*}
\tilde \mu _t ^{\varepsilon_i}(\phi) := \frac{1}{\sigma}\int _{\Omega} \phi(x) 
\varepsilon_i |\nabla \varphi ^{\varepsilon_i} (x,t)|^2 dx,
\quad \text{and} \quad
\hat \mu _t ^{\varepsilon_i}(\phi) := \frac{1}{\sigma}\int _{\Omega} \phi(x) \frac{2W (\varphi ^{\varepsilon_i} (x,t))}{\varepsilon_i} dx
\end{equation*}
for any $\phi \in C_c (\Omega)$.
\end{definition}

\begin{remark}\label{rem4.2}
If there exist $t\geq 0$ and a Radon measure $\mu _t$ on $\Omega$
such that 
\begin{equation}
\int _\Omega \left | \frac{\varepsilon  |\nabla \varphi  ^{\varepsilon } (x,t) |^2}{2} 
-\frac{W(\varphi  ^{\varepsilon } (x,t))}{\varepsilon} \right | \, dx \to 0
\label{eq5.3}
\end{equation} 
and $\mu _t ^\varepsilon \to \mu _t $ as Radon measures, namely, 
$$
\int _\Omega \phi \, d\mu _t ^{\varepsilon} \to \int _\Omega \phi \, d\mu _t, 
\qquad \text{for any} \ \phi \in C_c (\Omega),
$$
then $\tilde \mu _t ^\varepsilon$ and $\hat \mu _t ^\varepsilon$ also converge to $\mu _t$ as Radon measures.
\end{remark}

By the definition of the initial data $\varphi _0 ^\varepsilon$, we obtain
\begin{proposition}[Proposition 1.4 of \cite{ilmanen1993}]
\label{prop3.2}
$\sup _{i \in \N} \mu _0 ^{\varepsilon_i} (\Omega ) <\infty$.
Moreover $\mu _0 ^{\varepsilon_i} \to \mathscr{H}^{d-1} \lfloor_{M_0}$ as Radon measures.
\end{proposition}

Set $D_1 = \sup _{i \in \N} \mu _0 ^{\varepsilon_i} (\Omega )$. 
Proposition \ref{prop3.2} implies $D_1 <\infty$.
The integration by parts implies the following standard estimates: 
\begin{proposition}\label{prop3.3}
Let $\varphi ^{\varepsilon}$ be a solution to \eqref{ac}. Then we have
\begin{equation}
\frac{d}{dt} \int _\Omega \frac{\varepsilon  |\nabla \varphi  ^{\varepsilon } |^2}{2} 
+\frac{W(\varphi  ^{\varepsilon })}{\varepsilon} \, dx
+ \frac12 \int _\Omega \varepsilon \left( -\Delta \varphi ^{\varepsilon }  
+\frac{W' (\varphi ^{\varepsilon})}{\varepsilon^2} \right)^2 \, dx
\leq \frac{ d^2 }{R_0 ^2} \int _\Omega  
\frac{W(\varphi  ^{\varepsilon })}{\varepsilon} \, dx,
\label{eq3.7}
\end{equation}
\begin{equation}
\frac{d}{dt} \int _\Omega \frac{\varepsilon  |\nabla \varphi  ^{\varepsilon } |^2}{2} 
+\frac{W(\varphi  ^{\varepsilon })}{\varepsilon} \, dx
+ \frac12 \int _\Omega \varepsilon (\varphi  ^{\varepsilon } _t)^2 \, dx
\leq \frac{ d^2 }{R_0 ^2} \int _\Omega  
\frac{W(\varphi  ^{\varepsilon })}{\varepsilon} \, dx,
\label{eq3.9}
\end{equation}
and
\begin{equation}
\int _\Omega \frac{\varepsilon  |\nabla \varphi  ^{\varepsilon } (x,t) |^2}{2} 
+\frac{W(\varphi  ^{\varepsilon } (x,t))}{\varepsilon} \, dx
\leq D_1 e^ {\frac{ d^2 }{R_0 ^2} t}.
\label{eq3.8}
\end{equation}

\end{proposition}
\begin{proof}
By the integration by parts and Young's inequality, we have
\begin{equation*}
\begin{split}
&\frac{d}{dt} \int _\Omega \frac{\varepsilon  |\nabla \varphi  ^{\varepsilon } |^2}{2} 
+\frac{W(\varphi  ^{\varepsilon })}{\varepsilon} \, dx
=\int _\Omega \varepsilon \left( -\Delta \varphi ^{\varepsilon }  
+\frac{W' (\varphi ^{\varepsilon})}{\varepsilon^2} \right) \varphi ^\varepsilon _t \, dx\\
=&\int _\Omega \varepsilon \left( -\Delta \varphi ^{\varepsilon }  
+\frac{W' (\varphi ^{\varepsilon})}{\varepsilon^2} \right) 
\left( \Delta \varphi ^{\varepsilon }  
-\frac{W' (\varphi ^{\varepsilon})}{\varepsilon^2} 
+ g^\varepsilon \frac{\sqrt{2W (\varphi ^{\varepsilon}) } }{\varepsilon}
\right) 
\, dx \\
=&-\int _\Omega \varepsilon \left( -\Delta \varphi ^{\varepsilon }  
+\frac{W' (\varphi ^{\varepsilon})}{\varepsilon^2} \right) ^2
\, dx
+
 \int _\Omega \varepsilon \left( -\Delta \varphi ^{\varepsilon }  
+\frac{W' (\varphi ^{\varepsilon})}{\varepsilon^2} \right) 
g^\varepsilon \frac{\sqrt{2W (\varphi ^{\varepsilon}) } }{\varepsilon}
\, dx\\
\leq & - \frac12 \int _\Omega \varepsilon \left( -\Delta \varphi ^{\varepsilon }  
+\frac{W' (\varphi ^{\varepsilon})}{\varepsilon^2} \right) ^2
\, dx
+
\int _\Omega (g^\varepsilon)^2 \frac{W (\varphi ^{\varepsilon})  }{\varepsilon}
\, dx.
\end{split}
\end{equation*}
By this and $\sup_{x \in \Omega} |g| \leq \frac{d}{R_0}$ we obtain \eqref{eq3.7} and \eqref{eq3.8}.
Similarly we can obtain \eqref{eq3.9}.
\end{proof}

\bigskip

Next we show the monotonicity formula. Set
\[
\rho _{y,s} (x,t) = \frac{1}{(4\pi (s-t))^{\frac{d-1}{2}}} e^{-\frac{|x-y|^2 }{4(s-t)}}, 
\qquad
x,y \in \mathbb{R}^d, \ 0 \leq t < s<\infty. 
\]
Similar to the proof in \cite[p.\,2028]{takasao2017}, we obtain
the following monotonicity formula.
\begin{proposition}
Let $\varphi ^{\varepsilon_i}$ be a solution to \eqref{ac} with initial data 
$\varphi ^{\varepsilon _i}$ which satisfies \eqref{initial}
and $\mu _t ^{\varepsilon_i}$ be a Radon measure defined in \eqref{eq5.1}.
Then we have
\begin{equation}
\begin{split}
\frac{d}{dt} \int _{\mathbb{R}^d} \rho _{y,s} (x,t) \, d \mu _t ^\varepsilon (x)
\leq & \,
\frac{1}{2(s-t)} \int _{\mathbb{R}^d} \rho _{y,s} (x,t) 
\left( \frac{\varepsilon  |\nabla \varphi  ^{\varepsilon } (x,t) |^2}{2} 
-\frac{W(\varphi  ^{\varepsilon } (x,t))}{\varepsilon} \right) \, dx \\
& \, +\frac{d^2}{2R_0 ^2} 
\int _{\mathbb{R}^d} \rho _{y,s} (x,t) \, d \mu _t ^\varepsilon (x) .
\end{split}
\label{eq3.10}
\end{equation}
Here, $\mu _t ^\varepsilon$ is extended periodically to $\mathbb{R}^d$.
\end{proposition}

\bigskip

For the solution $\varphi ^\varepsilon$ to \eqref{ac},
under the parabolic change of variables $\tilde x =\frac{x}{\varepsilon}$ 
and $\tilde t = \frac{t}{\varepsilon^2}$, we have
\begin{equation}
\tilde \varphi ^{\varepsilon}_{\tilde t} 
=\Delta_{\tilde x} \tilde \varphi ^{\varepsilon} 
-W' (\tilde \varphi ^{\varepsilon})
+ \varepsilon \tilde g^{\varepsilon} \sqrt{2W(\tilde \varphi ^{\varepsilon})}, 
\quad \tilde x \in \Omega_\varepsilon, \ \tilde t>0,
\label{eq6.1}
\end{equation}
where $\Omega _\varepsilon =( \mathbb{R} /\varepsilon ^{-1} \mathbb{Z})^d $, 
$\tilde \varphi ^\varepsilon (\tilde x,\tilde t) = \varphi ^\varepsilon (x,t)$, 
$\tilde g ^\varepsilon (\tilde x,\tilde t) = g ^\varepsilon (x,t)$, 
and 
\begin{equation}
\| \tilde g^{\varepsilon }\|_{L^\infty} \leq \frac{d}{R_0}, 
\quad
\| \nabla_{\tilde x} \tilde g^{\varepsilon} \| _{L^\infty} \leq M,
\quad
\| \nabla_{\tilde x} ^2 \tilde g^{\varepsilon } \|_{L^\infty} \leq M.
\label{eq6.2}
\end{equation}
Therefore, the external force term 
$\varepsilon \tilde g^{\varepsilon} \sqrt{2W( \tilde \varphi ^{\varepsilon})}$
can be regarded as a small perturbation and 
we can obtain the following:
\begin{lemma}
For the solution $\varphi^\varepsilon$ to \eqref{ac}, there exists a constant $c>0$ depending
only on $d,\,M,\,D_1,$ and $T$ such that
\begin{equation}
\sup _{\Omega \times [\varepsilon ^2 ,T)} \varepsilon |\nabla \varphi ^\varepsilon | \leq c
\quad
\text{and}
\quad
\sup _{\Omega \times [\varepsilon ^2 ,T)} \varepsilon^2 |\nabla^2 \varphi ^\varepsilon | \leq c
\label{eq6.3}
\end{equation}
for any $\varepsilon >0$.
\end{lemma}

\begin{proof}
For the rescaled solution to \eqref{eq6.1} with \eqref{eq6.2}, the standard parabolic argument implies
the interior estimates of $|\nabla_{\tilde x} \tilde \varphi ^\varepsilon|$
and $|\nabla ^2 _{\tilde x} \tilde \varphi ^\varepsilon|$ (see \cite{MR0241822} and \cite[Lemma 4.1]{takasao-tonegawa}).
Hence we obtain \eqref{eq6.3}.
\end{proof} 


\section{Subsolution and supersolution}
We construct simple subsolutions and supersolutions
to \eqref{ac}
that represent obstacles.
In this section, we extend $\Omega$, $O_{\pm}$, and
the solution $\varphi ^{\varepsilon}$
periodically to $\mathbb{R}^d$. 
Set $\underline{r_y} (x) =\dfrac{1}{2R_0} (R^2 _0 -|x-y| ^2) $,
$\underline{\varphi ^\varepsilon _y} (x) =q^\varepsilon (\underline{r_y} (x))$ and
$\overline{\varphi ^\varepsilon _y} (x) =-q^\varepsilon (\underline{r_y} (x))
=q^\varepsilon (-\underline{r_y} (x))$ on $\mathbb{R}^d$.
\begin{lemma}\label{lem4.1}
Assume that $B_{R_0} (y) \subset O_{+}$. 
Then there exists $\epsilon _1 =\epsilon_1 (d,R_0) >0$  such that
$\underline{\varphi ^\varepsilon _y}$ is a subsolution to \eqref{ac} 
with $\mathbb{R}^d$ instead of $\Omega$,
for any $\varepsilon \in(0,\epsilon_1)$.
\end{lemma} 

\begin{proof}
Without loss of generality we assume $y=0$.
Let $\varphi ^\varepsilon$ be a solution to \eqref{ac} with $\mathbb{R}^d$ instead of $\Omega$.
By Remark \ref{rem3.1} and \eqref{q}, we have
\begin{equation*}
\begin{split}
\sqrt{2W(q ^\varepsilon )} r_t ^\varepsilon 
= 
\varepsilon q_r ^\varepsilon r_t ^\varepsilon
&=
\varepsilon q_r ^\varepsilon \Delta r^\varepsilon
+ \varepsilon q_{rr} ^\varepsilon |\nabla r ^\varepsilon | ^2
-\dfrac{W' (q ^\varepsilon )}{\varepsilon } 
+ g^{\varepsilon} \sqrt{2W(q ^\varepsilon )}\\
&=
\sqrt{2W(q ^\varepsilon )} \Delta r^\varepsilon
+ \dfrac{W' (q ^\varepsilon )}{\varepsilon } ( |\nabla r ^\varepsilon | ^2 -1) 
+ g^{\varepsilon} \sqrt{2W(q ^\varepsilon )}.
\end{split}
\end{equation*}
Thus the first equation \eqref{ac} 
with $\mathbb{R}^d$ instead of $\Omega$ is equivalent to
\begin{equation}
r_t^\varepsilon - \Delta r^\varepsilon + \frac{2q^\varepsilon (r^\varepsilon)}{\varepsilon} 
(|\nabla r^\varepsilon |^2 -1) - g^\varepsilon =0, 
\qquad (x,t) \in \mathbb{R}^d \times (0,\infty),
\label{ac2}
\end{equation}
where we used $W' (q^\varepsilon) /\sqrt{2W(q^\varepsilon)} = -2 q^\varepsilon$.
Therefore we only need to prove
\begin{equation}
- \Delta \underline{r_0} + \frac{2q^\varepsilon (\underline{r_0})}{\varepsilon} 
(|\nabla \underline{r_0} |^2 -1) - g^\varepsilon \leq 0, 
\qquad (x,t) \in \mathbb{R}^d \times (0,\infty)
\label{eq:4.2}
\end{equation}
for sufficiently small $\varepsilon >0$. We compute that
\begin{equation}
- \Delta \underline{r_0} + \frac{2q^\varepsilon (\underline{r_0})}{\varepsilon} 
(|\nabla \underline{r_0} |^2 -1) -  g^\varepsilon 
= \frac{d}{R_0} + \frac{2q^\varepsilon (\underline {r_0})}{\varepsilon} 
\left(\frac{ |x|^2 }{R^2 _0} -1 \right) - g^\varepsilon .
\label{eq:4.3}
\end{equation}

First we consider the case of $|x| \leq R_0$. We compute that
\begin{equation}
\frac{d}{R_0} + \frac{2q^\varepsilon (\underline {r_0})}{\varepsilon} 
\left(\frac{ |x|^2 }{R^2 _0} -1 \right) -g^\varepsilon
\leq \frac{d}{R_0} - \frac{d}{R_0} = 0,
\label{eq:4.4}
\end{equation}
where we used $\frac{2q^\varepsilon (\underline {r_0})}{\varepsilon} 
\left(\frac{ |x|^2 }{R^2 _0} -1 \right) \leq 0$
and $g^\varepsilon (x) =\frac{d}{R_0} $.
Therefore \eqref{eq:4.3} and \eqref{eq:4.4} imply \eqref{eq:4.2}
if $|x| \leq R_0$.

Next we consider the case of $R_0 \leq |x| \leq R_0 + \sqrt{\varepsilon}$.
Note that 
$\frac{2q^\varepsilon (\underline {r_0})}{\varepsilon} 
\left(\frac{ |x|^2 }{R^2 _0} -1 \right) \leq 0$ 
and $g^\varepsilon (x) =\frac{d}{R_0} $ also hold in this case.
Hence we obtain \eqref{eq:4.2} if $R_0 \leq |x| \leq R_0 + \sqrt{\varepsilon}$.

Finally we consider the case of $ R_0 + \sqrt{\varepsilon} \leq |x|$. 
We compute
\begin{equation*}
\begin{split}
\frac{2q^\varepsilon (\underline {r_0})}{\varepsilon} 
\left(\frac{ |x|^2 }{R^2 _0}-1 \right) 
\leq & \,
\frac{2}{\varepsilon} \tanh \left( - \frac{ 2R_0 \sqrt{\varepsilon} +\varepsilon }{2R_0 \varepsilon} \right) 
\frac{ 2R_0 \sqrt{\varepsilon} +\varepsilon }{R_0 ^2} \\
\leq & \, -\frac{4}{\sqrt{\varepsilon}R_0 } \tanh \Big(\frac{1}{\sqrt{\varepsilon}}\Big).
\end{split}
\end{equation*}
Therefore we have
\begin{equation*}
\begin{split}
\frac{d}{R_0} + \frac{2q^\varepsilon (\underline {r_0})}{\varepsilon} 
\left(\frac{ |x|^2 }{R^2 _0} -1 \right) - g^\varepsilon 
& 
\leq 2 \frac{d}{R_0}
 -\frac{4}{\sqrt{\varepsilon}R_0 } \tanh \Big(\frac{1}{\sqrt{\varepsilon}}\Big),
\end{split}
\end{equation*}
where we used $\max _{x\in \mathbb{R}^d}|g^\varepsilon(x) | \leq \frac{d}{R_0}$.
Thus there exists $\epsilon_1 =\epsilon_1 (d, R_0) >0$ such that \eqref{eq:4.2} holds
for any $\varepsilon \in(0,\epsilon_1)$.
\end{proof}

Similarly, we obtain
\begin{lemma}\label{lem4.2}
Assume that $B_{R_0} (y) \subset O_{-}$. 
Then $\overline{\varphi ^\varepsilon _y}$ 
is a supersolution to \eqref{ac} with $\mathbb{R}^d$ instead of $\Omega$,
for any $\varepsilon \in(0,\epsilon_1)$, where
$\epsilon_1$ is as in Lemma \ref{lem4.1}.
\end{lemma} 

In order to use the comparison principle, we need the following estimates for the initial data.
\begin{lemma}\label{lem4.3}
Assume that $B_{R_0} (y) \subset O_{+}$ and $B_{R_0} (z) \subset O_{-}$. 
Then 
\begin{equation}
\underline{\varphi ^{\varepsilon_i} _y} (x) \leq \varphi ^{\varepsilon_i} _0 (x) 
\qquad \text{and} \qquad
\overline{\varphi ^{\varepsilon_i} _z} (x) \geq \varphi ^{\varepsilon_i} _0 (x) ,
\qquad x \in \mathbb{R}^d
\label{eq:4.5}
\end{equation}
for sufficiently large $i\geq 1$.
\end{lemma} 

\begin{proof}
To show the first inequality of \eqref{eq:4.5},
we only need to prove that
\begin{equation}
\underline{r_y} (x) \leq \tilde r_0 ^{\varepsilon_i} (x) , 
\qquad x \in \mathbb{R}^d
\label{eq:4.6}
\end{equation}
for sufficiently large $i \geq 1$.

We assume that for any $N\geq 1$ there exist $i \geq N $ 
and  $x' \in \Omega$ such that 
$\underline{r_y} (x') > \tilde r _0 ^{\varepsilon_i} (x')$ 
with $\underline{r_y} (x') \geq 0$. 
In addition, we may assume that
$\sup_{x} |r _0 ^{\varepsilon_i} (x) -\tilde r_0 ^{\varepsilon_i} (x) | <\frac{\delta_1}{4}$.
Then $x' \in \overline{ B_{R_0} (y)}$
and 
\begin{equation}
0\leq \underline{r_y} (x') \leq \dist (x', \partial B_{R_0} (y))
\label{eq:4.7}
\end{equation}
by $\max _{x \in \overline{B_{R_0} (y)}} |\nabla \underline{r_y} (x) | \leq 1 $
and $\underline{r_y} =0$ on $\partial B_{R_0} (y)$.
The assumptions $B_{R_0} (y) \subset O_{+} \subset U_0 ^{i}$ and
\eqref{initialdata3} imply
\begin{equation}
\dist (x', \partial B_{R_0} (y)) + \frac{\delta _1}{2} \leq \dist (x',M_0^i) .
\label{eq:4.8}
\end{equation}
Then \eqref{eq:4.7} and \eqref{eq:4.8} imply
\begin{equation*}
\begin{split}
\underline{r_y} (x') \leq & \, \dist (x', \partial B_{R_0} (y)) 
\leq \dist (x',M_0^i) - \frac{\delta _1}{2}\\
= & \, r_0 ^{\varepsilon _i} (x') - \frac{\delta _1}{2}
< \tilde r_0 ^{\varepsilon _i} (x')  .
\end{split}
\end{equation*}
This is a contradiction to 
$\underline{r_y} (x') > \tilde r_0 ^{\varepsilon_i} (x')$.
In the case of $\underline{r_y} (x') < 0$, we may obtain a contradiction similarly,
by using $\min _{x \in (B_{R_0} (y)) ^c} |\nabla \underline{r_y} (x) | = 1 $. 
Therefore we obtain \eqref{eq:4.6}.
We can show the second inequality of \eqref{eq:4.5} by the similar argument.
\end{proof}

By Lemma \ref{lem4.1}, Lemma \ref{lem4.2}, and Lemma \ref{lem4.3}, we have
\begin{proposition}\label{prop4.4}
Assume that $\varphi ^{\varepsilon}$ is a solution to \eqref{ac},
$B_{R_0} (y) \subset O_{+}$ and $B_{R_0} (z) \subset O_{-}$. 
Then 
\begin{equation}
\underline{\varphi ^\varepsilon _y} (x) \leq \varphi ^\varepsilon (x, t) 
\leq \overline{\varphi ^\varepsilon _z} (x)
\quad \text{for any} \ (x,t) \in \mathbb{R}^d \times [0,\infty)
\label{eq:4.9}
\end{equation}
for sufficiently small $\varepsilon >0$.
\end{proposition}
For the proof, we only need to use the standard comparison principle. 
Therefore we omit it.


\section{Existence of weak solution to mean curvature flow with obstacles}
In this section, we prove the global existence of the weak solution
to the mean curvature flow with obstacles in the sense of Brakke's mean curvature flow.

\begin{theorem}\label{thm5.3}
Let $d= 2$ or $3$ and $\{ \varepsilon _i \} _{i=1} ^\infty$ 
be a positive sequence that converges to $0$.
Assume that $M_0$ and $O_\pm$ satisfy all the assumptions in Section 2.
Let $\varphi ^{\varepsilon_i}$ be a solution to \eqref{ac} with initial data 
$\varphi ^{\varepsilon _i} _0$ which satisfies \eqref{initial}
and $\mu _t ^{\varepsilon_i}$ be a Radon measure defined in \eqref{eq5.1}.
Then there exist a subsequence $\{ \varepsilon _{i_j} \} _{j=1} ^\infty$,
a family of Radon measures $\{ \mu_t \}_{t \in [0,\infty)}$, and
$\psi  \in BV_{loc} (\Omega \times [0,\infty) 
\cap C^{\frac12} _{loc} ([0,\infty) ;L^1 (\Omega))$
such that the following hold:
\begin{enumerate}
\item[{\rm (1)}] $\mu _0 = \mathscr{H}^{d-1} \lfloor_{M_0}$.
\item[{\rm (2)}] For any $t \in [0,\infty)$, $\mu _t ^{\varepsilon_{i_j}}$ converges to $\mu _t$ as Radon measures.
\item[{\rm (3)}] For a.e. $t\geq 0$, $\mu _t$ is $(d-1)$-integral.
\item[{\rm (4)}] $\psi =0$ or $1$ a.e. on $\Omega \times [0,\infty)$, and 
$\varphi ^{\varepsilon _i} \to 2\psi -1$ in $L^1 _{loc} (\Omega \times (0,\infty))$ and a.e. pointwise.
In addition, $\psi (\cdot,0) =\chi _{U_0}$ a.e. on $\Omega$ and
$\| \nabla \psi (\cdot,t)\| (\phi) \leq \mu _t (\phi)$ 
for all $t \in [0,\infty)$ and $\phi \in C_c(\Omega;[0,\infty))$.
\item[{\rm (5)}] $\spt \, \mu _t  \cap O_\pm=\emptyset$ for any $t \geq 0$,
and $\psi =1$ a.e. on $O_+ \times [0,\infty)$
and $\psi =0$ a.e. on $O_- \times [0,\infty)$.
\item[{\rm (6)}] $\{ \mu _t\} _{t \in [0,\infty)}$ is a Brakke's mean curvature flow 
on $\Omega \setminus \overline{O_+ \cup O_- }$, that is,
\begin{equation}
\int _\Omega \phi \, d\mu _t \Big|_{t=t_1} ^{t_2} 
\leq 
\int _{t_1} ^{t_2} \int _\Omega \{ (\nabla \phi -\phi h) \cdot h + \phi_t \} \, d\mu _t dt  
\label{eq5.8}
\end{equation}
for all $0\leq t_1 <t_2 <\infty$ and
$\phi \in C_c ^1 (\Omega \setminus \overline{O_+ \cup O_- }\times [0,\infty); [0,\infty))$.
\item[{\rm (7)}] 
As an additional assumption, suppose that $D_0$ used in \eqref{initialdata1} satisfies $D_0 <2$.
Then there exists $T_1 >0$ such that $\|\nabla \psi (\cdot,t)\| =\mu_t$ for a.e. $t \in [0,T_1)$.
\end{enumerate}
\end{theorem}
\begin{remark}
If $M_0$ is $C^1$, then the additional assumption of (7) holds (see Remark \ref{remark2.2}).
\end{remark}
\begin{remark}\label{remark5.3}
In a weak sense, $\tilde U(t) :=\{ x\in \Omega \, | \, \psi (x,t)=1 \}$ corresponds to 
$U_t$ in Section 1 when $\|\nabla \psi (\cdot,t)\| =\mu_t$.
This $\tilde U (t)$ has similar properties to the weak solution treated in 
\cite[Theorem 4.6]{almeida-chambolle-novaga}. 
More precisely, $\tilde U(t)$ is a Caccioppoli set for any $t\geq 0$ by 
$\| \nabla \psi (\cdot,t)\| (\Omega) \leq \mu _t (\Omega) <\infty$ 
and for any $T>0$ there exists $C_2>0$ defined below such that
$| \tilde U (t_2) \bigtriangleup \tilde U (t_1) | \leq C_2 \sqrt{t_2-t_1} $ for any 
$0 \leq t_1 <t_2 <T$ by 
$\psi \in C^{\frac12} _{loc} ([0,\infty) ;L^1 (\Omega))$ and $\psi =0$ or $1$ a.e. 
on $\Omega \times [0,\infty)$.
\end{remark}
\begin{remark}\label{remark5.4}
In the case of $d=2$ or $3$,
thanks to \cite{roger-schatzle} and \cite{mugnai-roger2011}, 
we can prove the integrality of $\mu _t$ and the Brakke's inequality by 
the standard energy estimates(see \cite{sato2008}). 
In contrast,
considering \cite{takasao-tonegawa, tonegawa2003},  
the pointwise estimate of 
$\left(\frac{\varepsilon  |\nabla \varphi  ^{\varepsilon } |^2}{2} 
-\frac{W(\varphi  ^{\varepsilon })}{\varepsilon} \right)_+$
and the parabolic monotonicity formula
seem to be important when $d\geq 4$.
\end{remark}

\begin{proof}
(1) holds by Proposition \ref{prop3.2}. 
In the case of $d=2$ or $3$, 
Proposition 4.3 and Proposition 4.4 in \cite{mugnai-roger2011} imply that if 
\begin{equation}
\sup _{i \in \N} \int _0 ^T \int _{\Omega} \frac{1}{\varepsilon_i}
\left(  g^{\varepsilon_i} \sqrt{2W(\varphi ^{\varepsilon _i})} \right) ^2
\, dxdt <\infty,
\label{eq5.2}
\end{equation}
then there exist a subsequence $\varepsilon_i \to 0$ 
(denoted by the same index) and 
a family of Radon measures $\{ \mu_t \}_{t \in [0,\infty)}$
such that (2) and (3) hold.
From \eqref{eq3.8} we have
$$
\int _0 ^T \int _{\Omega} \frac{1}{\varepsilon_i}
\left(  g^{\varepsilon_i} \sqrt{2W(\varphi ^{\varepsilon _i})} \right) ^2
\, dxdt
\leq
\frac{2 D_1  e^ {\frac{ d^2 }{R_0 ^2} T} d^2 T}{R_0 ^2}, \qquad i\geq 1.
$$
Therefore, \eqref{eq5.2} holds and we obtain (2) and (3).

Next we prove (4). Note that the proof is almost the same 
as that in \cite[Proposition 8.3]{takasao-tonegawa}.
Set
\[
\Phi (s) = \sigma ^{-1} \int_{-1} ^s \sqrt{2W(a)} \, da,
\quad 
\text{and}
\quad
w^{\varepsilon_i } = \Phi \circ \varphi ^{\varepsilon_i }.
\]
We remark that $\Phi (1) =1$ and $\Phi (-1) =0$. We compute
\begin{equation}
|\nabla w^{\varepsilon_i}|
=
\sigma^{-1} |\nabla \varphi ^{\varepsilon_i}| \sqrt{2W(\varphi ^{\varepsilon_i})}
\leq 
\sigma^{-1} 
\left(
\frac{\varepsilon _i |\nabla \varphi  ^{\varepsilon_i } |^2}{2} 
+\frac{W(\varphi  ^{\varepsilon_i })}{\varepsilon_i} 
\right)
\label{eq5.9}
\end{equation}
and
\[
|w^{\varepsilon_i} _t|
\leq 
\sigma^{-1} 
\left(
\frac{\varepsilon_i  |\varphi  ^{\varepsilon_i } _t |^2}{2} 
+\frac{W(\varphi  ^{\varepsilon_i })}{\varepsilon_i} 
\right).
\]
Thus, by \eqref{eq3.9} and \eqref{eq3.8}, there exists $C_1 =C_1 (d,R_0,D_1,W,T)>0$ such that
\[
\max _{0\leq t \leq T} \int _{\Omega} |\nabla w^{\varepsilon_i} (x,t) | \, dx
+
\int _0 ^T \int _{\Omega} |w^{\varepsilon_i} _t | \, dxdt \leq C_1
\]
for any $i \geq 1$. Therefore $\{ w^{\varepsilon_i} \} _{i=1} ^\infty$
is bounded in $BV_{loc} (\Omega \times [0,T])$.
The compactness theorem and a diagonal argument imply
that there exist a subsequence (denoted by the same index) and
$w \in BV_{loc} (\Omega \times [0,\infty))$
such that
\[
w^{\varepsilon _i} \to w 
\qquad
\text{strongly in } L^1 _{loc} (\Omega \times [0,\infty))
\] 
and a.e. pointwise. We define $\psi = (1 + \Phi^{-1} \circ w)/2$.
Then
\[
\varphi^{\varepsilon _i} \to 2\psi -1 
\qquad
\text{strongly in } L^1 _{loc} (\Omega \times [0,\infty))
\] 
and a.e. pointwise. Note that by
$ \sup _{i} \int _{\Omega} \frac{W (\varphi ^{\varepsilon _i})}{\varepsilon _i} <\infty$ for any $t\geq 0$,
$\varphi ^{\varepsilon _i} \to \pm 1$ for a.e. $(x,t)$, and hence
$\psi =1$ or $0$ for a.e. $(x,t)$. Note that we can easily check that 
$\psi =w$ on $\Omega\times [0,\infty)$ and $\psi \in BV_{loc} (\Omega \times [0,\infty))$.
For $0\leq t_1 <t_2 <T$, there exists $C_2 =C_2 (d,R_0,D_1,W,T)>0$ such that
\begin{equation}
\begin{split} 
\int _{\Omega} |w ^{\varepsilon _i} (x,t_2) -w ^{\varepsilon _i} (x,t_1)| \, dx
\leq & \,
\int _{\Omega} \int _{t_1} ^{t_2} |w ^{\varepsilon _i} _t| \, dt dx\\
\leq & \, \sigma ^{-1}
\int _{\Omega} \int _{t_1} ^{t_2} 
\left(
\frac{\varepsilon_i  |\varphi _t ^{\varepsilon_i } |^2}{2} \sqrt{t_2 -t} 
+\frac{W(\varphi  ^{\varepsilon_i })}{\varepsilon_i \sqrt{t_2 -t} } 
\right)
 \, dt dx \\
\leq & C _2 \sqrt{t_2 -t_1} .
\end{split} 
\label{eq5.7}
\end{equation} 
By \eqref{eq5.7} and
\[
\lim _{i\to \infty} \int _{\Omega} |w ^{\varepsilon _i} (x,t_2) -w ^{\varepsilon _i} (x,t_1)| \, dx
= \int _{\Omega} |\psi (x,t_2) -\psi (x,t_1)| \, dx,
\]
we obtain $\psi \in C^{\frac12} _{loc} ([0,\infty) ;L^1 (\Omega))$.
In addition, \cite[Proposition 1.4]{ilmanen1993} yields $\psi (\cdot,0) =\chi _{U_0}$ a.e. on $\Omega$,
and \eqref{eq5.9} and $\|\nabla \psi (\cdot,t)\| =\|\nabla w (\cdot,t)\|$ imply 
$\| \nabla \psi (\cdot,t)\| (\phi) \leq \mu _t (\phi)$ 
for all $t \in [0,\infty)$ and $\phi \in C_c(\Omega;[0,\infty))$.
Therefore we obtain (4).

Now we show (5). 
In order to obtain $\spt \mu _t \cap O_+ =\emptyset$, we only need to prove that
\begin{equation}
\mu _t ^{\varepsilon_i} (\overline{B_r (y)}) = 
\frac12 \tilde \mu _t ^{\varepsilon_i} (\overline{B_r (y)}) 
+ \frac12 \hat \mu _t ^{\varepsilon_i} (\overline{B_r (y)}) \to 0
\qquad \text{as} \ i \to \infty
\label{eq5.4}
\end{equation}
for any $\overline{B_r (y)} \subset O_+$ with $0<r<R_0$.
Assume that $\overline{B_r (y)} \subset O_+$.
First we show $\hat \mu _t ^{\varepsilon_i} (\overline{B_r (y)}) \to 0$.
Let $z \in \R^d$ satisfy $\overline{B_r (y)} \subset B_{R_0} (z)$. Then
$\underline{\varphi _z ^{\varepsilon_i}}  \to 1$ uniformly on $\overline{B_r (y)}$,
since $\min _{x \in \overline{B_r (y)}} \underline{r_z} (x) >0$.
In addition, Proposition \ref{prop3.4} and \eqref{eq:4.9} imply
$\underline{\varphi _z ^{\varepsilon_i}} \leq \varphi ^{\varepsilon_i} \leq 1$.
Therefore $\hat \mu _t ^{\varepsilon_i} (\overline{B_r (y)}) \to 0$. 

To prove $\tilde \mu _t ^{\varepsilon_i} (\overline{B_r (y)}) \to 0$,
we suppose that $t>0$ (in the case of $t=0$, the claim is obvious).
Let $\phi \in C_c ^\infty (O_+ )$ be an non-negative test function. 
It is enough to show $\tilde \mu _t ^{\varepsilon_i} (\phi) \to 0$.
We may assume that $\spt \phi \subset B_{R_0} (z)$ for some $z \in O_+$.
By the integration by parts, we have
\begin{equation}
\begin{split}
\tilde \mu _t ^{\varepsilon_i} (\phi) 
= & \frac{\varepsilon_i}{\sigma}\int _{\spt \phi} \phi
\nabla (\varphi ^{\varepsilon_i}-1)  \cdot \nabla \varphi ^{\varepsilon_i} \,  dx\\
= & - \frac{\varepsilon_i}{\sigma}\int _{\spt \phi} \left(
\phi (\varphi ^{\varepsilon_i}-1)  \Delta \varphi ^{\varepsilon_i} 
+ (\varphi ^{\varepsilon_i}-1) \nabla \phi \cdot \nabla \varphi ^{\varepsilon_i} 
\right) \, dx.
\end{split}
\label{eq5.5}
\end{equation}
By \eqref{eq:4.9}, Proposition \ref{prop3.4}, and $\min _{x \in \spt \phi } \underline{r_z} (x) >0$,
there exists $C_3>0$ such that 
\begin{equation*}
\tanh (C_3/\varepsilon_i) \leq \varphi ^{\varepsilon_i} (x,t) < 1, \qquad x \in \spt \phi.
\end{equation*}
Therefore
\begin{equation}
| \varphi ^{\varepsilon_i} (x,t) - 1 | 
\leq 1-\tanh (C_3/\varepsilon _i ) \leq \varepsilon_i ^{2}, \qquad x \in \spt \phi
\label{eq5.6}
\end{equation}
for sufficiently large $i\geq 1$.
By \eqref{eq6.3}, \eqref{eq5.5}, and \eqref{eq5.6}, we obtain 
$\tilde \mu _t ^{\varepsilon_i} (\phi) \to 0$.
Hence we obtain \eqref{eq5.4}, and consequently $\spt \mu _t \cap O_+ =\emptyset$.
The other case 
($\spt \mu _t \cap O_- =\emptyset$)
and the remained claims
may be proved similarly.

Next we show (6). 
Given arbitrary open set $U \subset \Omega \setminus \overline{O_+ \cup O_- }$,
$\varphi ^{\varepsilon _i}$ is a solution to
$$
\varepsilon_i \varphi ^{\varepsilon_i}_t 
=\varepsilon_i \Delta \varphi ^{\varepsilon_i } 
-\dfrac{W' (\varphi ^{\varepsilon_i })}{\varepsilon_i } 
$$
on $U$. Then \cite[Proposition 4.5]{liu-sato-tonegawa} 
(with transport term $u\equiv 0$) tells us that
$\mu_t $ satisfies Brakke's inequality \eqref{eq5.8} on $U$
(see also \cite{ilmanen1993, mugnai-roger2008, mugnai-roger2011, sato2008}). 
Nevertheless, for the convenience of the reader, we prove (6) here.
By the integration by parts, we have
\begin{equation*}
\begin{split}
& \mu _{t_2} ^{\varepsilon _i} (\phi) -\mu _{t_1} ^{\varepsilon _i} (\phi) \\
= & \, \int_{t_1} ^{t_2} \left(
\frac{1}{\sigma} \int _U (- \phi \varepsilon_i ^{-1} (w^{\varepsilon _i} )^2
+\nabla \phi \cdot \nabla \varphi^{\varepsilon _i} w^{\varepsilon _i} )\, dx 
+ \mu _t ^{\varepsilon _i} (\phi _t) \right)dt
\end{split}
\end{equation*}
for $\phi \in C_c ^1 (U\times [0,\infty);[0,\infty))$, where 
$w ^{\varepsilon _i} = - \varepsilon_i \Delta \varphi ^{\varepsilon_i } 
+\dfrac{W' (\varphi ^{\varepsilon_i })}{\varepsilon_i } $.
By \cite[Theorem 4.3]{mugnai-roger2008},
\[
\int _{s_1} ^{s_2} \int_{V} |h|^2 \, d\mu _t dt \leq 
\liminf _{i \to \infty} 
\int _{s_1} ^{s_2} \int_{V} \varepsilon ^{-1} _i (w^{\varepsilon _i})^2 \, dx dt
\]
holds for any open set $V \times (s_1 ,s_2) \subset U \times [t_1,t_2]$.
Therefore we have
\begin{equation}
\int _{t_1} ^{t_2} \int_{U} \phi |h|^2 \, d\mu _t dt \leq 
\liminf _{i \to \infty} 
\int _{t_2} ^{t_2} \int_{U} \phi \varepsilon ^{-1} _i (w^{\varepsilon _i})^2 \, dx dt.
\label{eq5.10}
\end{equation}
In addition, \cite[Lemma 7.1]{mugnai-roger2008} implies
\begin{equation}
\int _{t_1} ^{t_2} \int_{U} \nabla \phi \cdot h \, d\mu _t dt 
= 
\lim _{i \to \infty} 
\int _{t_1} ^{t_2} \int_{U} \nabla \phi \cdot \nabla \varphi^{\varepsilon _i} w^{\varepsilon _i} \, dx dt.
\label{eq5.11}
\end{equation}
By \eqref{eq5.10}, \eqref{eq5.11}, and (2), \eqref{eq5.8} holds on $U$.
Therefore $\{ \mu _t \} _{t\in [0,\infty)}$
is a Brakke's mean curvature flow on $\Omega \setminus \overline{O_+ \cup O_- }$.

Finally we prove (7). From (3), for a.e. $t \geq 0$, there exists a $(d-1)$-rectifiable set $M_t$
and $\theta _t :M_t \to \mathbb{N}$ such that
$\mu _t = \theta _t \mathscr{H}^{d-1} \lfloor _{M_t}$. We only need to prove that 
$\{ \theta _t \geq 2 \}$ has measure zero for a.e. $t \in [0,T_1)$ for a suitable $T_1 >0$
(see \cite[p.\,275]{liu-sato-tonegawa} and \cite[p.\,926]{takasao-tonegawa}). 
We will determine $T_1$ in the following.
By Proposition \ref{prop3.3}, we have
\[
\sup_{i \in \mathbb{N}}
\mu_t ^{\varepsilon_i} (\Omega) 
<\infty 
\quad 
\text{and}
\quad
\sup_{i \in \mathbb{N}}
\int _ 0^T \int _\Omega \varepsilon_i \left( -\Delta \varphi ^{\varepsilon_i }  
+\frac{W' (\varphi ^{\varepsilon_i})}{\varepsilon^2 _i} \right)^2 \, dx dt <\infty
\]
for any $t\geq 0$ and $T>0$.
Hence Proposition 6.1 in \cite{mugnai-roger2008} implies that
there exists a subsequence $\varepsilon _i \to 0$ (denoted by the same index)
such that
\begin{equation}
\lim _{i\to \infty}
\int _ 0^T \int _\Omega \left|
\frac{\varepsilon _i |\nabla \varphi  ^{\varepsilon_i } |^2}{2} 
- \frac{W(\varphi  ^{\varepsilon_i })}{\varepsilon_i} 
\right| \, dx dt =0
\label{eq5.12}
\end{equation}
for $2\leq d \leq 3$.
By \eqref{eq3.10} and \eqref{eq5.12}, we have
\begin{equation}
\int _{\mathbb{R}^d} \rho _{y,s} (x,t) \, d \mu _t  (x)
\leq
 e^{\frac{d^2 }{2R_0 ^2} t}
 \int _{\mathbb{R}^d} \rho _{y,s} (x,0) \, d \mu _0 (x)
\label{eq5.13}
\end{equation}
for any $0 \leq t <s$ and $y \in \mathbb{R}^d$.
Assume that there exists $x_0$, $t_0$, and $N\geq 2$ such that
$\theta _{t_0} (x_0) =N$ and $\lim _{r\to 0} \dfrac{\mu _{t_0} (B_r (x_0))}{\omega _{d-1} r^{d-1} } =N$.
For $r>0$ and $a>0$, we compute
\begin{equation*}
\begin{split}
& \, \int _{B_{ar}(x_0)} \rho _{x_0 ,t_0 +r^2} (x,t_0) \,d\mu _{t_0} (x) \\
= & \, 
\frac{1}{(4\pi r^2 )^{\frac{d-1}{2}}} 
\int _{B_{ar}(x_0)} e^{-\frac{|x-x_0|^2}{4r^2}} \,d\mu _{t_0} (x) \\
= & \, 
\frac{1}{(4\pi r^2 )^{\frac{d-1}{2}}} 
\int _0 ^1 \mu _{t_0} (\{ x \in B_{ar}(x_0) \, | \, e^{-\frac{|x-x_0|^2}{4r^2}} >k \}) \, dk \\
= & \, 
\frac{1}{(4\pi r^2 )^{\frac{d-1}{2}}} 
\int _{e^{-\frac{a^2}{4}}} ^1 \mu _{t_0} (B_{\sqrt{4r^2\log \frac{1}{k}}} (x_0) ) \, dk \\
\to & \,  \frac{N \omega_{d-1}}{\pi ^{\frac{d-1}{2}}} \int _{e^{-\frac{a^2}{4}}} ^1 \left( \log \frac{1}{k} \right)^{\frac{d-1}{2}} \, dk \ \ \text{as} \ r\to \infty.
\end{split}
\end{equation*}
Note that 
$
\int _0 ^1 \left( \log \frac{1}{k} \right)^{\frac{d-1}{2}} \, dk
=\Gamma (\frac{d-1}{2} +1)
=\pi ^{\frac{d-1}{2}} /\omega_{d-1}.
$
Therefore
\begin{equation}
\lim _{r\to 0} \int _{\mathbb{R}^d} \rho _{x_0 ,t_0 +r^2} (x,t_0) \,d\mu _{t_0} (x) \geq N. 
\label{eq5.14}
\end{equation}
By \eqref{initialdata1} with $D_0 <2$, there exists $T_1 \in (0,1)$ depending only on $M_0$ such that
\begin{equation}
\int  _{\mathbb{R}^d} \rho _{y,s} (x,0) \, d \mu _0 (x) <2
\label{eq5.15}
\end{equation} 
for any $(y,s) \in \mathbb{R}^d \times (0,T_1]$.
Then we would have a contradiction from \eqref{eq5.13}, 
\eqref{eq5.14}, and \eqref{eq5.15}.
Therefore we obtain (7).
\end{proof}

\section*{Acknowledgement}
The author expresses his gratitude to the referees for the helpful comments
and suggestions that helped him to improve the original manuscript.
This work was supported by JSPS KAKENHI
Grant Numbers JP20K14343, JP18H03670, 
and JSPS Leading Initiative for Excellent Young Researchers (LEADER) operated by Funds for the Development of Human Resources in Science and Technology.

\bibliography{references}

\end{document}